\documentclass[11pt]{article}

\usepackage{amsthm}
\usepackage{enumitem}
\newtheorem{theorem}{Theorem}

\newtheorem{lemma}[theorem]{Lemma}
\newtheorem{proposition}[theorem]{Proposition}
\newtheorem{corollary}[theorem]{Corollary}
\newtheorem{definition}[theorem]{Definition}
\newcommand{\R}{\mathbb{R}}
\newcommand{\pr}{\mathbb{P}}

\newcommand{\E}{{\mathbb{E}}}
\newcommand{\N}{{\mathbb{N}}}
\newcommand{\erf}{\operatorname{erf}}
\newcommand{\erfc}{\operatorname{erfc}}

\newcommand{\ind}{\mathbf{I}}

\usepackage{amsmath} 
\usepackage{amssymb}
\usepackage{xcolor}
\usepackage{hyperref} 
\usepackage{graphicx}
\begin{document}

\title{Special times for the point set process of the Brownian Net}

\date{}

\author{Ruibo Kou
\\
Mathematics Institute, University of Warwick,
Coventry, CV4 7AL, UK}

\maketitle
\begin{abstract}
It is known that the point set process of the Brownian net is almost surely locally finite for all deterministic time, and there are random times that break this locally finiteness property. It is shown in this paper that the set of such random times has Hausdorff dimension $\frac{1}{2}$ almost surely.
\end{abstract}
\noindent
\section{Introduction and statement of result}
\subsection{Introduction}
The Brownian net $\mathcal{N}$ is an interacting particle system of one dimensional coalescing branching Brownian motions starting from every point in space–time $\mathbb{R}^2$. It was first constructed by Sun and Swart in \cite{BN} via the left-right Brownian web. Newman, Ravishankar and Schertzer \cite{BN2} later showed that allowing branching at $(1,2)$ points in the Brownian web yields the same object, namely, the Brownian net. Recent work has focused on the structure of the associated point set process. For $A\subseteq\R$ define the point set process $\xi_t^A$ by
\begin{equation*}
    \xi^A_t:=\{\pi(t):\pi\in\mathcal{N}(A\times\{0\})\},
\end{equation*}
the set of particle positions at time $t$ whose trajectories originate from $A \times \{0\}$. (See Section \ref{Defs and Notations} for the formal definition of $\mathcal{N}$ and corresponding notations.) In fact, the process $(\xi_t^A)_{t\geq0}$ is a Markov process (see \cite[Section.1.9]{BN}). Moreover, it was shown by Garrod, Tribe and Zaboronski that for each deterministic time $t>0$, $\xi_t^A$ is a Pfaffian point process. (See \cite[Section 4(b)]{Pff} for details.) Hence $\xi^A_t$ is almost surely locally finite at any deterministic time. However, it was shown in \cite[Proposition 3.14]{SP} that almost surely there exists a random dense set $T$ such that $\xi^A_t$ contains cluster points for all $t\in T$. We refer to times in $T$ as \emph{special times}. The main reason for the existence of such special times is the infinite branching rate of particles in $\xi_t^A$. In any time interval, each particle gives rise to infinitely many new particles. 

The primary goal of this paper is to determine the Hausdorff dimension of the random set $T$. Specifically, we shall establish the following Theorem:
\begin{theorem}\label{main}
    Let $T:=\{t:\xi_t^{\R}\text{ contains cluster points}\}$. Then almost surely $\dim T=\frac{1}{2}$.
\end{theorem}
\textbf{Remark}. This Theorem can be generalized. For any nonempty set $A\subseteq\R$ the set of special time $T(A):=\{t:\xi^A_t \text{ contains cluster points}\}$ has Hausdorff dimension $\frac{1}{2}$. Indeed, on the one hand, we have $T(A)\subseteq T$ so $\dim T(A)\leq\dim T\leq\frac{1}{2}$. On the other hand, as shown in Section \ref{Lower bound}, the set of special times along a single path (see the definition of $T^1_l$ and $T^2_r$ in the beginning of Section \ref{Upper bound} for details.) starting at $z\in A\times\{0\}$ has Hausdorff dimension at least $\frac{1}{2}$. Hence we have $\dim T(A)\geq\frac{1}{2}$.

The rest of this paper is organized as follows. In the rest of the introduction we recall one of constructions of the Brownian net  via the left-right Brownian web in \cite{BN}, together with properties of its special points  established in \cite{SP} . In Section \ref{Upper bound}, we derive an upper bound of the Hausdorff dimension using these properties. In Section \ref{Lower bound} we obtain a lower bound by applying the limsup fractal method introduced in \cite[Chapter 10]{PM}.

\subsection{Construction of the Brownian Net through the left-right Brownian web}\label{Defs and Notations}
We now recall one construction of the Brownian net in \cite{BN}. The Brownian web and Brownian net are random variables taking values in the space of path collections equipped with a corresponding sigma algebra. So to define such random variables, we firstly provide the space that a path lives in. For any two points $(x_1,t_1),(x_2,t_2)\in\R^2$, define the following metric:
\begin{equation*}
    \rho((x_1,t_1),(x_2,t_2))=|\tanh(t_1)-\tanh(t_2)|\vee\left|\frac{\tanh(x_1)}{1+|t_1|}-\frac{\tanh(x_2)}{1+|t_2|}\right|
\end{equation*}
and let $\overline{\mathbb{R}^2}$ be the closure of $\R^2$ under the metric $\rho$. The space $(\overline{\R^2},\rho)$ can be treated as identifying all points in $[-\infty,\infty]\times\{\infty\}$ (resp. $[-\infty,\infty]\times\{-\infty\}$) as a single point $(*,\infty)$ (resp. $(*,-\infty)$) 

Let a map $\pi:[\sigma_{\pi},\infty]\to[-\infty,\infty]\cup\{*\}$ be such that $\pi(\infty)=*$ and $\pi(\sigma_{\pi})=*$ whenever $\sigma_{\pi}=-\infty$. Suppose further that $t\mapsto(\pi(t),t)$ is continuous under the metric $\rho$ then we call it a path starting at time $\sigma_{\pi}$. Let $\Pi$ be the collection of all paths endowed with the following metric:
\begin{equation*}
    d(\pi_1,\pi_2)=|\tanh(\sigma_{\pi_1})-\tanh(\sigma_{\pi_2})|\vee\sup_{t\geq\sigma_{\pi_1}\wedge\sigma_{\pi_2}}\left|\frac{\tanh(\pi_1(t\vee\sigma_{\pi_1}))}{1+|t|}-\frac{\tanh(\pi_2(t\vee\sigma_{\pi_2}))}{1+|t|}\right|,
\end{equation*}
which makes the space $(\Pi,d)$ being complete and separable. 

We now let $\mathcal{H}$ be the collection of all compact subsets of $(\Pi,d)$ endowed with  the Hausdorff metric $d_{\mathcal{H}}$ given by
\begin{equation*}
    d_{\mathcal{H}}(K_1,K_2)=\sup_{\pi_1\in K_1}\inf_{\pi_2\in K_2} d(\pi_1,\pi_2)\vee\sup_{\pi_2\in K_2}\inf_{\pi_1\in K_1} d(\pi_1,\pi_2).
\end{equation*}
Let $\mathcal{B}_\mathcal{H}$ be the Borel sigma algebra with respect to $d_\mathcal{H}$, then the Brownian web and net are $(\mathcal{H},\mathcal{B_H})$ random variables.

In the rest of this paper, we adapt the same notations as in \cite{SP}. For any $\pi\in\Pi$, we denote the starting time of $\pi$ by $\sigma_\pi$. If a path $\pi$ starts at a deterministic point $z\in\R^2$, we denote this path by $\pi_z$. For any $K\in\mathcal{H}$ and $A\subseteq\overline{\R^2}$, we let $K(A)$ to be the set of paths with starting point in $A$.
If $A=\{z\}$ is a singleton, we shall use $K(z)$ to denote all path starting at $z$. 

\begin{definition}
    We call $(l_1,\cdots,l_n;r_1,\cdots,r_m)$ a collection of left-right coalescing Brownian motions if $(l_1,\cdots,l_n)$ is distributed as coalescing Brownian motions with drift $-1$ and $(r_1,\cdots,r_m)$ is distributed as coalescing Brownian motions with drift $+1$ and any pair $(l_i,r_j)$ is distributed as the unique weak solution of the following stochastic differential equation:
\begin{equation}\label{left-right BM}
    \begin{split}
        dL_t=\mathbf1_{\{L_t\neq R_t\}}dB^l_t+\mathbf{1}_{\{L_t=R_t\}}dB_t^s-dt,\\
        dR_t=\mathbf1_{\{L_t\neq R_t\}}dB^r_t+\mathbf{1}_{\{L_t=R_t\}}dB_t^s+dt,
    \end{split}
\end{equation}
under the constraint that 
$$L_t\leq R_t,\text{ for all }t>\tau_{L,R}, $$
where $B^l$, $B^r$, $B^s$ are three independent standard Brownian motions and $\tau_{L,R}$ is the first meeting time of $L$ and $R$. 
\end{definition}
 The existence and uniqueness of the weak solution of (\ref{left-right BM}) follows by showing $D_t=\frac{R_t-L_t}{\sqrt{2}}$ is a reflected sticky Brownian motion with drift. In fact, the transition density $P_t(0,y)$ of $D_t$ starting at zero is given by
\begin{equation}\label{transition density}
    \begin{split}
        P_t(0,y)&=2\sqrt{2}\left(1+\sqrt{2}y+2t\right)e^{2\sqrt{2}y}\erfc(\frac{y}{\sqrt{2t}}+\sqrt{t})-\frac{4\sqrt{2t}}{\sqrt{\pi}}e^{-\frac{y^2}{2t}-t+\sqrt{2}y}\\
        &+\left((1+2t)\erfc(\sqrt{t})-2e^{-t}\sqrt{\frac{t}{\pi}}\right)\delta_0(y),\ y\geq 0, t\geq 0,
    \end{split}
\end{equation}
where $\delta_0(y)$ is the Dirac delta function. We leave this computation in the Appendix. The main idea is the same as in \cite[Proposition 13]{SB} where the transition density of reflected sticky Brownian motion (without drift) is computed.
Now we have the following characterization of the left-right Brownian web from \cite[Theorem 1.5]{BN}.
\begin{theorem}\label{The Brownian web}
    There exists a $(\mathcal{H}^2, \mathcal{B}_{\mathcal{H}^2})$-valued random variable $(\mathcal{W}^l, \mathcal{W}^{\mathrm{r}})$, called the standard left-right Brownian web, whose distribution is uniquely determined by the following properties:
    \begin{enumerate}[label=(\alph*)]
        \item For each deterministic $z \in \mathbb{R}^2$, almost surely there are unique paths $l_z \in \mathcal{W}^l(z)$ and $r_z \in \mathcal{W}^{r}(z)$.
        \item For any finite deterministic set of points $z_1, \ldots, z_n, z_1^{\prime}, \ldots, z_m^{\prime} \in \mathbb{R}^2$, the collection $\left(l_{z_1}, \ldots, l_{z_n} ; r_{z_1^{\prime}}, \ldots, r_{z_m^{\prime}}\right)$ is distributed as left-right coalescing Brownian motions.
        \item For any deterministic countable dense subset $\mathcal{D} \subset \mathbb{R}^2$, almost surely $\mathcal{W}^l$ is the closure of $\mathcal{W}^l(\mathcal{D})$ and $\mathcal{W}^{r}$ is the closure of $\mathcal{W}^r(\mathcal{D})$ in the space $(\Pi, d)$.
    \end{enumerate}
\end{theorem}

\textbf{Remark.} We say two paths $\pi_1,\pi_2\in\Pi$ crosses each other during time interval $(s,t)$ if there exists $u,v\in(s,t)$ such that $\sigma_{\pi_1}\vee\sigma_{\pi_{2}}<s<u<v<t$ and $(\pi_1(u)-\pi_2(u))(\pi_1(v)-\pi_2(v))<0$. Almost surely, for any two path $l_1,l_2\in\mathcal{W}^l$, $l_1$ can not cross $l_2$ during any time interval. This fact is followed by the classification of special points of the Brownian web. (See \cite[Theorem 3.11]{BW1} or \cite[Theorem 2.11]{BW3} for the classification of all points in $\R^2$ according to the Brownian web.) There are seven classes in total and non of them allow crossing. To be more precise, almost surely, for any $l_1,l_2\in\mathcal{W}^l$ and $s>\sigma_{l_1}\vee\sigma_{l_2}$ with $l_1(s)=l_2(s)$, we have $l_1=l_2$ on $[s,\infty)$. An analogous statement holds for $\mathcal{W}^r$. This fact will be used in next section.

For the left-right Brownian web, there is a unique pair of dual web $(\hat{\mathcal{W}}^l,\hat{\mathcal{W}}^r)$ contains random set of paths running backward in time. (See \cite{BW1,BW2} for the existence of the dual Brownian web where the web there is defined by coalescing Brownian motions without drift.) To give the definition of the dual left-right Brownian web, we start by introducing some notations. A path $\hat{\pi}$ running backward in time is a map $\hat{\pi}:[-\infty,\hat{\sigma}_{\hat{\pi}}]\to[-\infty,+\infty]\cup\{*\}$ that makes $t\mapsto(\hat{\pi}(t),t)$ being continuous under the metric $\rho$, where the starting time of $\hat{\pi}$ is denoted by $\hat{\sigma}_{\hat{\pi}}$. Similar as before, let $(\hat{\Pi},\hat d)$ be the space of all backward paths constructed through time reversal of the space $(\Pi,d)$. And let $\hat{\mathcal{H}}$ be the collection of all compact subsets of $\hat{\Pi}$. For two paths $(\pi,\hat\pi)\in(\Pi,\hat\Pi)$, we say they crosses each other if there exist $\sigma_\pi\leq s<t\leq \hat\sigma_{\hat\pi}$ such that $(\pi(s)-\hat\pi(s))(\pi(t)-\hat\pi(t))<0$. We recall the following characterization of the dual left-right Brownian web from \cite{BN}:

\begin{theorem}
    There is a $\hat{\mathcal{H}^2}$ random variable $(\hat{\mathcal{W}^l},\hat{\mathcal{W}^r})$ defined on the same probability space as $(\mathcal{W}^l,\mathcal{W}^r)$ that is uniquely determined by the following properties:
    \begin{enumerate}[label=(\alph*)]
        \item For any deterministic point $z\in\R^2$, almost surely $\hat{\mathcal{W}^l}$ (resp. $\hat{\mathcal{W}^r}$) consists of a single path $\hat{l}_z$ (resp. $\hat{r}_z$) starting at $z$ such that this path does not cross any path in $\mathcal{W}^l$. (resp. $\mathcal{W}^r$)
        \item For any deterministic countable dense subset $\mathcal{D}\subset \R^2$, almost surely $\hat{\mathcal{W}^l}$ (resp. $\hat{\mathcal{W}^r}$) is the closure of $\hat{\mathcal{W}^l}(\mathcal{D})$ (resp. $\hat{\mathcal{W}^r}(\mathcal{D})$) under metric $\hat d$.
    \end{enumerate}
    
\end{theorem}

\textbf{Hopping construction of Brownian net} (\cite[Theorem 1.3]{BN}): For any $\pi_1,\pi_2\in\Pi$, we say a time $t$ is an intersection time of $\pi_1$ and $\pi_2$ if $\pi_1(t)=\pi_2(t)$. And a path constructed by concatenating the piece of $\pi_1$ before and the piece of $\pi_2$ after an intersection time is called a path hopping from $\pi_1$ to $\pi_2$. The Brownian net $\mathcal{N}$ is the closure of all paths constructed by hopping finitely many times in $\mathcal{W}^l\cup \mathcal{W}^r$.

\subsection{Special points of the Brownian net}
Special points in the Brownian net are well studied in \cite{SP}. One way of classification of such points is to classify them according to incoming paths.

\begin{definition}
    For $z\in\R^2$ and $\pi\in\mathcal{N}$, we say $\pi$ is an incoming path at $z=(x,t)$ if $\pi(t)=x$ and $\sigma_\pi<t$. For a pair of left-right Brownian motion $(l,r)\in(\mathcal{W}^l,\mathcal{W}^r)$, we say it is an incoming pair at $z$ if $l(t)=r(t)=x$ and $\sigma_l\vee\sigma_r<t$.
\end{definition}
It is shown in \cite[Lemma 3.9(a)]{SP} that there are seven classes of special points in total and we recall three of them and some properties that is useful in our proof. 
\begin{definition}
   For any point $z\in\R^2$, we say it is a
   \begin{itemize}
       \item $C_n$ point if there is an incoming path $\pi\in\mathcal{N}$ at $z$ but no incoming path $\pi\in\mathcal{W}^l\cup\mathcal{W}^r$ at $z$,
       \item $C_l$ point if there is an incoming path $l\in\mathcal{W}^l$ at $z$ but no incoming path $r\in\mathcal{W}^r$ at $z$,
       \item $C_r$ point if there is an incoming path $r\in\mathcal{W}^r$ at $z$ but no incoming path $l\in\mathcal{W}^l$ at $z$.
   \end{itemize}
\end{definition}
The following Lemma (see \cite[Lemma 3.7]{SP}) is crucial because it provides an equivalent condition for a point being an cluster point.
\begin{lemma}\label{isolated point}
    Almost surely a point $x\in\xi^\R_t$ is isolated from the left in the set $\xi^\R_t$ if and only there is an incoming path $l\in\mathcal{W}^l$ at $(x,t)$. An analogous statement holds if $x$ is isolated from the right.  
\end{lemma}
\textbf{Remark}: According to Lemma \ref{isolated point},  $x\in\xi_t^\R$ is a cluster point isolated from the left (resp. right) if and only $(x,t)$ is a $C_l$ (resp. $C_r$) point. And $x\in\xi^\R_t$ is neither isolated from the left nor right if and only if $(x,t)$ is a $C_n$ point. This is the main reason we focus on these three classes of points.

The following Proposition (see \cite[Proposition 3.11]{SP}) describes what happen around $C_l$ and $C_r$ points through excursions.
\begin{proposition}\label{excursion along l}
If $z=(x, t)$ is a $C_l$ point, then there exist $l \in \mathcal{W}^l$ and $r_n \in \mathcal{W}^r (n \geq 1)$, such that $l(t)=x<r_n(t)$, each $r_n$ makes an excursion away from $l$ on a time interval $\left(s_n, u_n\right) \ni t$, $\left[s_n, u_n\right] \subset\left(s_{n-1}, u_{n-1}\right), s_n \uparrow t, u_n \downarrow t$, and $r_n(t) \downarrow x$.(see Figure \ref{Cl}.) By symmetry, an analogous statement holds for $C_r$ point.
\end{proposition}
\begin{figure}[htpb]
\centering
\includegraphics[scale=0.8]{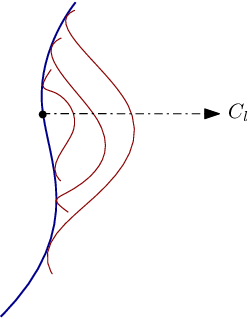}
\caption{Excursions along $l$}
\label{Cl}
\end{figure}

The last lemma in this section characterizes the point at a deterministic time:
\begin{lemma}\label{incoming pair}
     For any deterministic $t\in\R$:
     \begin{enumerate}[label=(\alph*)]
         \item Almost surely, for any point in $\R\times\{t\}$ either there is no incoming path at that point or there is an incoming pair $(l,r)\in (\mathcal{W}^l,\mathcal{W}^r)$ at that point and both cases occur.
         \item Let $S^t:=\{x\in\R:\text{There exists an incoming pair $(l,r)$ at $(x,t)$}\}$ and let $\mathcal{F}_t$ be the sigma algebra generated by $\{\left(\pi(s)\right)_{s\in[\sigma_\pi,t]}:\pi\in\mathcal{N},\sigma_\pi\leq t\}.$ Then almost surely for any $x\in S^t$ the process $D^{l,r}_t:=\frac{r_t-l_t}{\sqrt{2}}$ with $(l,r)$ being an incoming pair at $(x,t)$ has the following conditional distribution:
         \begin{equation*}
             \pr[D^{l,r}_{t+s}\in dy\bigm|\mathcal{F}_t]=P_s(0,y)dy,
         \end{equation*}
         where $P_s(0,y)$ is given by (\ref{transition density}).
     \end{enumerate}
\end{lemma}
\begin{proof}
    Part (a) follows from \cite[Lemma 3.9(b)]{SP}. So we only need to show (b). Given an order of all pairs $(l_n,r_n)$ in $(\mathcal{W}^l(\mathbb{Q}^2),\mathcal{W}^r(\mathbb{Q}^2))$ whose starting times satisfy $\sigma_{l_n}\vee\sigma_{r_n}<t$. Note that by \cite[Lemma 3.4]{BN}, almost surely there exists $k\in\N$ and $\epsilon>0$ such that $(l_k,r_k)=(l,r)$ on $[t-\epsilon,\infty)$. Let $N=\inf\{k:(l_k,r_k)=(l,r)\text{ on }[t-\epsilon,\infty)\}$ and for any $n\in\N$ set $D^n_{s}:=\frac{r_n(s)-l_n(s)}{\sqrt2}$ for $s\geq \sigma_{l_n}\vee\sigma_{r_n}$. So we have $D^N_s=D^{l,r}_s$ on $[t-\epsilon,\infty)$ almost surely. In particular $D^N_t=0$ almost surely. Let $P_s(x,y)$ be the transition density of $\frac{R_s-L_s}{\sqrt2}$, where $(L_s,R_s)$ solves (\ref{left-right BM}). For any bounded and measurable function $\varphi$ by Markov property and (b) of Theorem \ref{The Brownian web} we have 
    \begin{equation*}
        \begin{split}
            \E[\varphi(D^{l,r}_{t+s})\bigm|\mathcal{F}_t]&=\E\left[\varphi(D_{t+s}^N)\big|\mathcal{F}_t\right]\\
            &=\E\left[\sum_{n=1}^\infty\ind\{N=n\}\varphi(D_{t+s}^n)\bigm|\mathcal{F}_t\right]\\
            &=\sum_{n=1}^\infty\ind\{N=n\}\E\left[\varphi(D_{t+s}^n)\bigm|\mathcal{F}_t\right]\\
            &=\sum_{n=1}^\infty\ind\{N=n\}\int_0^\infty \varphi(y)P_{s}(D_{t}^n,y)dy\\
            &=\int_0^\infty\varphi(y)P_s(D_t^N,y)dy\\
            &=\int_0^\infty\varphi(y)P_s(0,y)dy.
        \end{split}
    \end{equation*}
    Hence we get almost surely $\pr[D^{l,r}_{t+s}\in dy\bigm|\mathcal{F}_t]=P_s(0,y)dy$ as desired.
\end{proof}

\section{Upper bound of the Hausdorff dimension}\label{Upper bound}
In this section, we aim at deriving an upper bound of the Hausdorff dimension of $T$, namely, we shall prove $\dim T\leq \frac{1}{2}$. Firstly, we introduce some notations. For any path $l\in \mathcal{W}^l(\mathbb{Q}^2)$, let
\begin{equation*}
    T^1_l=\{t:(l(t),t) \text{ is a $C_l$ point}\}.
\end{equation*}
Similarly, for any $r\in\mathcal{W}^r(\mathbb{Q}^2)$ and $\pi\in\mathcal{N}(\mathbb{Q}^2)$, we let
\begin{equation*}
    T_r^2=\{t:(r(t),t)\text{ is a $C_r$ point}\};\ T_\pi^3=\{t:(\pi(t),t)\text{ is a $C_n$ point}\}.
\end{equation*}
The first step is to provide upper bounds of the Hausdorff dimension of $T_l^1$, $T_r^2$ and $T_\pi^3$ respectively.

\subsection{Upper bounds of the Hausdorff dimension of $T_l^1$ and $T^2_r$}
Let $l\in\mathcal{W}^l(\mathbb{Q}^2)$ and choose $a,b\in\R$ such that $\sigma_l<a<b$.
The idea is to cover $T_l^1\cap [a,b]$ by intervals of the form $[t_k^j,t_k^{j+1}]$ for $k\in\mathbb{N}$ and $0\leq j\leq 2^k$, where $t_j^k:=(b-a)j2^{-k}+a$. For each $k\in\N$, let $J^1_k$ be the collection of intervals of the form $[t_k^j,t^{j+1}_k]$ such that 
\begin{equation*}
    r(t^j_k)-l(t^j_k)>0,
\end{equation*}
where $r\in \mathcal{W}^r$ is an incoming right path at $(l(t_k^{j-1}),t_k^{j-1})$. 
The existence of such a path is guaranteed by Lemma \ref{incoming pair}(a) because there is an incoming path $l$ at $(l(t_k^{j-1}),t_k^{j-1})$.
\begin{lemma}\label{cover of T_l}
    Almost surely, for arbitrarily small $\delta>0$, let $m=\inf\{k:2^{-k}<\delta\}$, the collection $J^1(\delta)=\cup_{k=m}^\infty J^1_k$ is a covering of the set $T_l^1\cap [a,b]$.
\end{lemma}
The following lemma that helps separating two points is needed for proving Lemma \ref{cover of T_l}.
\begin{lemma}\label{speration of close time}
    For any $\epsilon>0$ and $x_1,x_2\in[a,b]$ such that $0<x_2-x_1<\epsilon$, there exists $M\in\N$ and $0\leq j\leq M$ such that $2^{-M}<2\epsilon$ and 
    $x_1\in[t^{j-1}_M,t_M^j]$ and $x_2\in[t^j_M,t^{j+1}_M]$.
\end{lemma}
    \begin{figure}[htpb]
    \centering
    \includegraphics[scale=0.8]{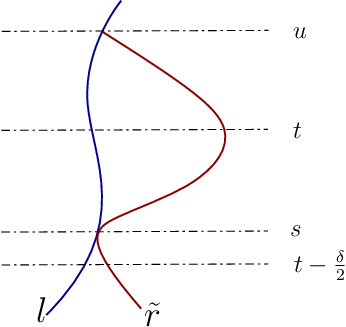}
    \caption{Existence of $l$ and $\tilde{r}$}
    \label{s and t}
    \end{figure}
We leave the proof of this Lemma to the appendix because it is just an analysis proof.
\begin{proof}(proof of Lemma \ref{cover of T_l})
    Fix a realization of the Brownian net. Let $t\in T^1_l\cap[a,b]$. For any small enough $\delta>0$, according to Proposition \ref{excursion along l}, there exists $s
    \in(t-\frac{\delta}{2},t)$ and an incoming right path $\tilde{r}\in\mathcal{W}^r$ at $(l(s),s)$, where $\tilde{r}$ makes an excursion away from $l$ on time interval $(s,u)$ with $u>t$. (See Figure \ref{s and t}.) By Lemma \ref{speration of close time}, we can find $M\in\N$ with $2^{-M}<\delta$ such that $s\in[t^{j-1}_M,t^j_M]$ and $t\in[t^{j}_M,t^{j+1}_M]$. Since $\tilde{r}$ makes an excursion away from $l$, we get $\tilde{r}(t^j_M)-l(t^j_M)>0$.
    By Lemma \ref{incoming pair}(a), there is an incoming right path $r\in \mathcal{W}^r$ at $(l(t^{j-1}_M),t^{j-1}_M)$. It is enough to prove $l(t^{j}_M)<r(t^{j}_M)$ because it follows that $t\in J^1_M\subseteq J^1(\delta)$ as $M\geq m$.
    Set 
    \begin{equation*}
    U^{l,\tilde{r}}_{s,j,M}:=\{(x,u):l(u)\leq x<\tilde{r}(u), s\leq u\leq t^{j}_M\}.
    \end{equation*}
    Note that according to Lemma \ref{incoming pair}(b) the transition density of $\frac{r-l}{\sqrt2}$ is given by (\ref{transition density}), which is supported on $[0,\infty)$. Hence $r$ can not cross $l$ during time interval $(t^{j-1}_M,\infty)$. So we have $l(s)\leq r(s)$. Suppose that $r(s)=l(s)=\tilde{r}(s)$, then $r=\tilde{r}$ on time interval $(s,\infty)$ due to the fact that $r$ and $\tilde{r}$ coalesce at $(r(s),s)$. So we get $r(t^j_M)=\tilde{r}(t^j_M)>l(t^j_M)$. Suppose that $r(s)>l(s)=\tilde{r}(s)$. It follows from the remark after Theorem \ref{The Brownian web} that $r$ can not cross $\tilde r$ during any time interval. Since $r$ is a continue path it can not enter $U^{l,\tilde{r}}_{s,j,M}$, that is, there is no $u\in[s,t^{j}_M]$ such that $(r(u),u)\in U^{l,\tilde{r}}_{s,j,M}$. (See Figure \ref{U along l} as an illustration.) In particular $l(t^j_M)<\tilde{r}(t^j_M)\leq r(t^{j}_M)$ as desired.
\end{proof}

\begin{figure}[htpb]
\centering
\includegraphics[scale=0.8]{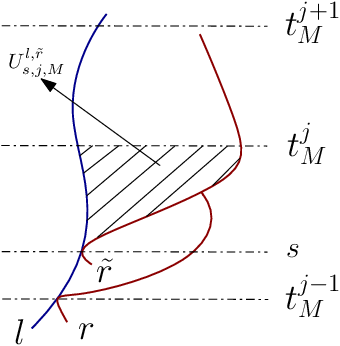}
\caption{$r$ can not enter $U^{l,\tilde{r}}_{s,j,M}$}
\label{U along l}
\end{figure}

\begin{lemma}\label{upperbound of T^1_l}
    The Hausdorff dimension $\dim T_l^1\cap[a,b]\leq \frac{1}{2}$ almost surely.
\end{lemma}

\begin{proof}
    It is suffices to show:
    \begin{equation*}
        \sup_{\delta>0}\E[\sum_{I\in J^1(\delta)}|I|^\gamma]<\infty.
    \end{equation*}
    For any $\delta>0$, according to Lemma \ref{incoming pair}(b) we have:
    \begin{equation*}
        \E[\sum_{I\in J^1(\delta)}|I|^\gamma]=\sum_{k=m}^\infty\sum_{j=1}^{2^k}\left((b-a)2^{-k}\right)^\gamma\int_0^\infty P_{(b-a)2^{-k}}(0,y)dy,
    \end{equation*}
    where $P_t(0,y)$ is the transition density given by (\ref{transition density}). To compute that expectation it is enough to consider the transition density for small $t$. In fact
    \begin{equation*}
    \begin{split}
        \int_0^\infty P_t(0,y)dy&=1-(1+2t)\erfc(\sqrt{t})+2e^{-t}\sqrt{\frac{t}{\pi}}\\        &=\erf(\sqrt{t})-2t\erfc(\sqrt{t})+2e^{-t}\sqrt{\frac{t}{\pi}}=O(\sqrt{t}),\ \text{for small }t  
    \end{split}
    \end{equation*}
    as $\erf(x)$ is of order $x$ when $x$ is small.
    Hence there exists a constant $C>0$ such that
    \begin{equation*}
        \E[\sum_{I\in J^1(\delta)}|I|^\gamma]\leq C\sum_{k=m}^\infty 2^{k(1-\gamma)}\sqrt{\frac{1}{2^k}}\leq C\sum_{k=1}^\infty 2^{k(\frac{1}{2}-\gamma)}<\infty
    \end{equation*}
    provided $\gamma>\frac{1}{2}$.
Note that the map $\delta\mapsto\E[\sum_{I\in J^1(\delta)}|I|^\gamma]$ is decreasing in $\delta$ so
\begin{equation*}
    \sup_{\delta>0}\E[\sum_{I\in J^1(\delta)}|I|^\gamma]<\infty
\end{equation*}
for any $\gamma>\frac{1}{2}$. Hence according to the definition of Hausdorff dimension, we get $\dim T_l^1\cap[a,b]\leq \frac{1}{2}$ almost surely.
\end{proof}

By symmetry, same upper bound holds for $\dim T^2_r\cap[a,b]$. 

\begin{corollary}
For any $l\in \mathcal{W}^l$ and $r\in \mathcal{W}^r$, let $T^1_l$ and $T^2_r$ be sets of times defined in the beginning of this section. Then almost surely $\dim T_l^1\leq\frac{1}{2}$ and $\dim T^2_r\leq\frac{1}{2}$.
\end{corollary}
\begin{proof}
    Let $\left([a_k,b_k]\right)_{k\geq0}$ be a sequence of closed intervals such that $\cup_{k\geq1}[a_k,b_k]=(\sigma_l,\infty)$. Then $T^1_l\cap(\sigma_l,\infty)=\bigcup_{k\geq1}(T^1_l\cap[a_k,b_k])$. Since each set $T^1_l\cap[a_k,b_k]$ has Hausdorff dimension at most $\frac{1}{2}$ by Lemma \ref{upperbound of T^1_l} and $\dim\left(\bigcup_{k\geq1}(T^1_l\cap[a_k,b_k])\right)=\sup_{k\geq1}\dim \left(T^1_l\cap[a_k,b_k]\right)\leq\frac{1}{2}$, we get $\dim\left(T^1_l\cap(\sigma_l,\infty)\right)\leq\frac{1}{2}$. Note that $T^1_l=\left(T^1_l\cap(\sigma_l,\infty)\right)\cup\{\sigma_l\}$, we have $\dim T^1_l\leq\frac{1}{2}$. Same upper bound holds for $\dim T^2_r$ by symmetry.
\end{proof}

\subsection{Upper bound of the Hausdorff dimension of $T^3_\pi$}
Let $\pi\in \mathcal{N}$, the idea of deriving an upper bound of $\dim T^3_\pi$ is the same. So firstly, we need a similar proposition as Proposition \ref{excursion along l} that describes what happen around $C_n$ point.
\begin{proposition}\label{excursion along pi}
    If $z=(x,t)$ is a $C_n$ point and $\pi\in\mathcal{N}$ is an incoming path at $z$. Then there exists a sequence $s_n\uparrow t$ such that for each $s_n$, there exists an incoming pair $(l_n,r_n)\in(\mathcal{W}^l,\mathcal{W}^r)$ with $l_n(s_n)=r_n(s_n)=\pi(s_n)$ and $l_n(u)\leq\pi(u)\leq r_n(u)$, $l_n(u)<r_n(u)$ for all $u\in(s_n,t]$.(See Figure \ref{Cn point})
\end{proposition}

    \begin{figure}[htpb]
    \centering
    \includegraphics[scale=0.8]{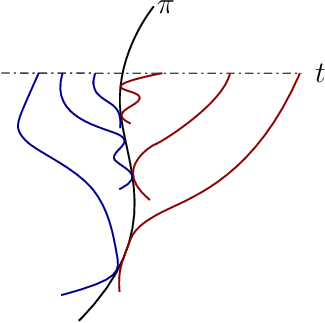}
    \caption{Left-right paths around $C_n$ point}
    \label{Cn point}
    \end{figure}

\begin{proof}
    Choose $(t_n)_{n\geq1}$ from $\mathbb{Q}$ such that $t_n\uparrow t$. Then by Lemma \ref{incoming pair}(a), there is an incoming pair $(l_n,r_n)$ at $(\pi(t_n),t_n)$. Note that there is no incoming left or right path at $z$ so we have $l_n(t)<\pi(t)<r_n(t)$. Let
    $s_n=\inf\{u:l_n(t-u)\neq r_n(t-u)\}$ be the last separation time of $l_n$ and $r_n$ before time $t$. Then $l_n(u)<r_n(u)$ for all $u\in(s_n,t]$. Moreover, according to \cite[Proposition 1.8]{BN}, $\pi$ is contained by $l_n$ and $s_n$. Hence we get $l_n\leq\pi\leq r_n$ on $[s_n,t]$. From the fact that all path in $\mathcal{N}$ is continue, we have $t_n\leq s_n<t$. Hence $s_n\uparrow t$ as $t_n\uparrow t$.
\end{proof}
With the help of Proposition \ref{excursion along pi}, we can now prove the upper bound of $\dim T_\pi^3$ by repeating the same argument.
\begin{lemma}
    For any $\pi\in\mathcal{N}(\mathbb{Q}^2)$, let $T_\pi^3$ be define as in the beginning of this section, then $\dim T^3_\pi\leq\frac{1}{2}$ almost surely.
\end{lemma}
\begin{proof}
    Choose $\sigma_\pi<a<b$ as before. For any interval of the form $[t^{j}_k,t^{j+1}_k]$, we place it into $J^3_k$ if and only if $l(t^{j}_k)<r(t^{j}_k)$, where $(l,r)$ is an incoming pair at $(\pi(t^{j-1}_k),t^{j-1}_k)$. For any $\delta>0$, let $m=\inf\{k:2^{-k}<\delta\}$ and $J^3(\delta)=\cup_{k=m}^\infty J^3_k$, we claim that $J^3(\delta)$ covers $T^3_\pi\cap[a,b]$ almost surely. By Proposition \ref{excursion along pi}, there exists $s\in(t-\frac{\delta}{2},t)$ and an incoming pair $(\tilde{l},\tilde{r})\in (\mathcal{W}^l,\mathcal{W}^r)$ at $(\pi(s),s)$ such that $\tilde{l}(t)<\pi(t)<\tilde{r}(t)$. According to Lemma \ref{speration of close time} we can find $M\geq m$ and $1\leq j\leq M$ such that $s\in[t^{j-1}_M,t^j_M]$ and $t\in[t^j_M,t^{j+1}_M]$. By Lemma \ref{incoming pair}(a), there is an incoming pair $(l,r)\in(\mathcal{W}^l,\mathcal{W}^r)$ at $(\pi(t^{j-1}_M),t^{j-1}_M)$. Similar as before, to see $J^3(\delta)$ covers $t$, it is enough to show $l(t^j_M)<r(t^j_M)$ so that we have $t\in J^3_M\subseteq J^3(\delta)$.
    We do this by showing $\tilde{r}(t^j_M)\leq r(t^j_M)$ firstly. 
    Without loss of generality, we assume $\pi(t^j_M)<\tilde{r}(t^j_M)$. Otherwise we get $r(t^j_M)\geq\pi(t^j_M)=\tilde{r}(t^j_M)$ because $\pi\leq\tilde{r}$ on $[s,t]$ by Proposition \ref{excursion along pi}. Let $s'=\inf\{u:\pi(t^j_M-u)=\tilde{r}(t^j_M-u)\}$ be the last separation time of $\pi$ and $\tilde{r}$ before time $t^j_M$. Then $s\leq s'<t^j_M$. Set
    \begin{equation*}
        U^{\pi,\tilde{r}}_{s',j,M}:=\{(x,u):\pi(u)\leq x<\tilde{r}(u),s'\leq u\leq t^j_M\}.
    \end{equation*}
    Combining \cite[Proposition 1.8]{BN} and the remark after Theorem \ref{The Brownian web}, we have $r$ can neither cross $\tilde{r}$ nor $\pi$ during time interval $[t^{j-1}_M,t^j_M]$. Hence it can not enter $U^{\pi,\tilde{r}}_{s',j,M}$. (See Figure \ref{U along pi} as an illustration.) In particular, $\tilde{r}(t^j_M)\leq r(t^j_M)$. By symmetry we get $l(t^j_M)\leq\tilde{l}(t^j_M)$. Hence $l(t^j_M)\leq\tilde{l}(t^j_M)<\tilde{r}(t^j_M)\leq r(t^j_M)$, which means $T_\pi^3\cap[a,b]$ is covered by intervals in $J^3(\delta)$.
    \begin{figure}[htpb]
\centering
\includegraphics[scale=0.8]{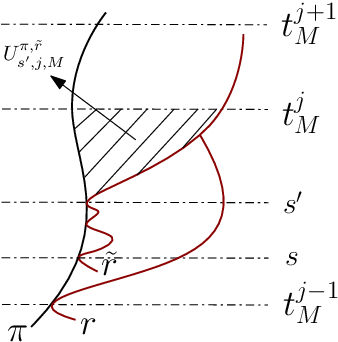}
\caption{$r$ can not enter $U^{\pi,\tilde{r}}_{s',j,M}$}
\label{U along pi}
\end{figure}
    
    The uniform boundedness of the first moment:
   \begin{equation*}
       \sup_{\delta>0}\E[\sum_{I\in J^3(\delta)}|I|^\gamma]<\infty
   \end{equation*}
   is followed by the same computation in the proof of Lemma \ref{upperbound of T^1_l} as the transition density of $\frac{r-l}{\sqrt2}$ remain the same. Hence $\dim T_\pi^3\leq\frac{1}{2}$.
\end{proof}

\subsection{Proof of upper bound in Theorem \ref{main}}
According to the remark after Lemma \ref{isolated point}, it is easy to see $T\subseteq\bigcup_{l\in \mathcal{W}^l}T^1_l\bigcup_{r\in \mathcal{W}^r}T^2_r\bigcup_{\pi\in\mathcal{N}}T_\pi^3$ and each of those set has Hausdorff dimension at most $\frac{1}{2}$. In fact, thanks to the following Lemma which can be found in \cite[Lemma 2.7 (a)]{SP}, we will show that it is enough to cover $T$ by countably many sets: $T\subseteq\bigcup_{l\in \mathcal{W}^l(\mathbb{Q}^2)}T^1_l\bigcup_{r\in \mathcal{W}^r(\mathbb{Q}^2)}T^2_r\bigcup_{\pi\in\mathcal{N}(\mathbb{Q}^2)}T_\pi^3$.
\begin{lemma}\label{dual pair}
    Almost surely, for each $\pi\in\mathcal{N}$ and $u>\sigma_\pi$, there exist  $\hat{l}\in\hat{\mathcal{W}}^l(\pi(u),u)$ and $\hat{r}\in\hat{\mathcal{W}^r}(\pi(u),u)$ such that $\hat r\leq\pi\leq\hat{l}$ and $\hat r<\hat l$ on $(\sigma_\pi,u)$.
\end{lemma}
\begin{lemma}
    Let $T$ be the set of special times defined in Theorem \ref{main}, then $\dim T\leq \frac{1}{2}$ almost surely.
\end{lemma}
\begin{proof}
    For any $t\in T$, suppose that $t\in T^3_\pi$ with the starting point $(\pi(\sigma_\pi),\sigma_\pi)$ of $\pi$ not being in $\mathbb{Q}^2$. Then by Lemma \ref{dual pair}, there exists $\hat{r}\in \hat{\mathcal{W}}^r(\pi(t),t)$  and $\hat{l}\in \hat{\mathcal{W}}^l(\pi(t),t)$ with $\hat{r}<\hat{l}$ on $(\sigma_\pi,t)$. 
    Note that forward and dual paths spend zero time together, which implies that $\int_{\sigma_\pi}^t\ind\{\pi(u)=\hat l(u)\}du=0$. (See \cite[Lemma 2.12]{SP} for details.) We get $\mathbb{Q}^2\cap\{(x,u):\pi(u)<\hat{l}(u),\sigma_\pi<u<t\}\neq\emptyset$. Choose $w\in\mathbb{Q}^2\cap\{(x,u):\pi(u)<\hat{l}(u),\sigma_\pi<u<t\}$ and let $l_w\in\mathcal{W}^l(w)$. Then the meeting time $\tau$ of $l_w$ and $\pi$ is less equal than $t$ because $l_w$ can not cross $\hat l$. (See Figure \ref{dual} for an illustration.) Now define a new path $\tilde{\pi}$ starting at $w$ by
    \begin{equation}\label{new path}
    \tilde{\pi}=
    \begin{cases}
         l_w\ &\text{on }[\sigma_w,\tau]\\
         \pi\ &\text{on }(\tau,\infty)
         \end{cases}
    \end{equation}
Since both $l_w$ and $\pi$ are path in $\mathcal{N}$ and $\tilde{\pi}$ is the coalesce path of $l_w$ and $\pi$, we get $\tilde{\pi}\in\mathcal{N}$. (See \cite[Proposition 3.15(i)]{BW3} for details) Hence $t\in T^3_{\tilde{\pi}}$ with $\tilde{\pi}$ starting at a rational point. Suppose that $t\in T^1_l$, replacing $\pi$ by $l$ and repeating the same method, the new path $\tilde{\pi}$ defined in (\ref{new path}) is a path in $\mathcal{W}^l(w)$ with $w\in\mathbb{Q}^2$. So $t\in T^1_{\tilde{\pi}}$. An analogue result holds when $t\in T^2_r$. Hence we get $T\subseteq\bigcup_{l\in W^l(\mathbb{Q}^2)}T^1_l\bigcup_{r\in W^r(\mathbb{Q}^2)}T^2_r\bigcup_{\pi\in\mathcal{N}(\mathbb{Q}^2)}T_\pi^3$. Since each of these sets has Hausdorff dimension at most $\frac{1}{2}$, a countable union of them still has Hausdorff dimension at most $\frac{1}{2}$.
\end{proof}
    \begin{figure}[htpb]
    \centering
    \includegraphics[scale=0.8]{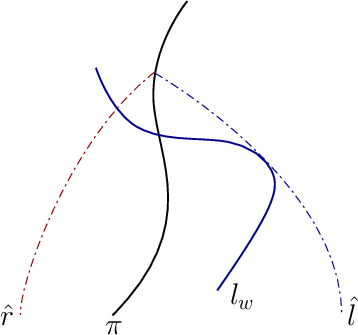}
    \caption{Existence of dual paths, where dashed lines are dual paths}
    \label{dual}
    \end{figure}

\section{Lower bound of the Hausdorff dimension}\label{Lower bound}
To derive a lower bound of the Hausdorff dimension of $T$ in Theorem \ref{main}, we need to recall the definition of limsup fractal. Let $l\in\mathcal{W}^l(\mathbb{Q}^2)$ and $a,b$ be two real numbers such that $\sigma_l<a<b$.
For any positive integer $n$, let $\mathcal{C}_n$ be collection of intervals of the form $[t^j_n,t^{j+1}_n]$ where $t^j_n=(b-a)j2^{-n}+a$ as before. Let $\mathcal{C}=\cup_{n>0}\mathcal{C}_n$.
\begin{definition}
    For any interval $I\in\mathcal{C}$, we associate a Bernoulli random variable $Z(I)$ to $I$, and let $A(k)=\cup_{I\in \mathcal{C}_k,Z(I)=1}I$. We define $A$ to be the limsup fractal associated with the family $\{Z(I):I\in\mathcal{C}\}$ by
    \begin{equation*}
        A=\bigcap_{n=1}^\infty\bigcup_{k=n}^\infty A(k).
    \end{equation*}
\end{definition}
We shall see that as a subset of $T$, $T^1_l$ defined in the beginning of section 2 contains a limsup fractal and the lower bound will be deduced by using the following proposition whose proof can be found in \cite[Theorem 2.1]{LSF} or \cite[Theorem 10.6]{PM}. 
\begin{proposition}\label{PM book}
    Suppose that $\{Z(I):I\in\mathcal{C}\}$ is a collection of Bernoulli random variables such that $p_n:=\pr[Z(I)=1]$ is the same for all $I\in\mathcal{C}_n$. For $I\in\mathcal{C}_m$ with $m<n$, define
    \begin{equation*}
        M_n(I):=\sum_{J\in\mathcal{C}_n,J\subset I}Z(J).
    \end{equation*}
    Suppose there is a $\gamma\in(0,1)$ such that
    \begin{enumerate}[label=(\alph*)]
        \item Var$(M_n(I))\leq p_n2^{n-m+1}$,
        \item $\lim_{n\to\infty}2^{n(\gamma-1)+1}p_n^{-1}=0,$
    \end{enumerate}
    then dim $A\geq\gamma$ almost surely for the limsup fractal $A$ associated with $\{Z(I):I\in\mathcal{C}\}$.
\end{proposition}
 For any $I=[t^j_n,t^{j+1}_n]\in\mathcal{C}_n$, by Lemma \ref{incoming pair}(a) there is an incoming path $r\in\mathcal{W}^r$ at $(l(t^{j-1}_n),t^{j-1}_n)$ and the distribution of $D_t=\frac{r_t-l_t}{\sqrt{2}}$ on $[t^{j-1}_n,\infty)$ is given by the transition density $P_t(0,y)$ whose expression is provided by (\ref{transition density}). We set $Z(I)=1$ (see also Figure \ref{Z(I)=1}) if and only if the pair $(l,r)$ satisfying 
\begin{equation*}
        \frac{1}{2^n}<\frac{r(t^j_{n})-l(t^{j}_n)}{\sqrt{2}}<\frac{1}{2^{n/4}}
\end{equation*}
and
\begin{equation*}
        \frac{r(t)-l(t)}{\sqrt{2}}\in(\frac{1}{2^n},\frac{1}{n})\ \text{for all }t\in[t^{j}_n,t^{j+1}_n].
\end{equation*}
    \begin{figure}[htpb]
    \centering
    \includegraphics[scale=0.8]{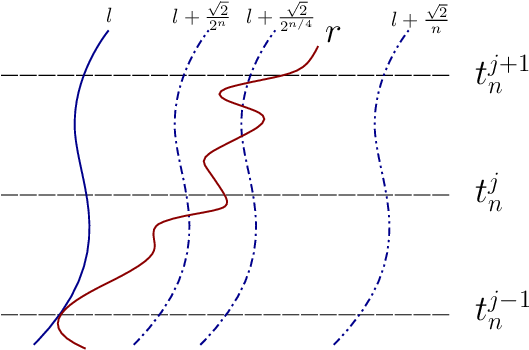}
    \caption{Strategy for letting $Z(I)=1$}
    \label{Z(I)=1}
    \end{figure}
Clearly, for each $n\in\N$, $Z(I)=1$ has the same probability for all $I\in\mathcal{C}_n$ due to the Markov property. So we can write 
\begin{equation*}
    p_n:=\pr[Z(I)=1]
\end{equation*}
for all $I\in\mathcal{C}_n$.
The first step is estimation of $p_n$ for large $n$. In the rest of this paper, for any two sequence $(a_n)_{n\geq1}$ and $(b_n)_{n\geq1}$, we say $a_n\approx b_n$ if  $\frac{a_n}{b_n}\to C$ as $n\to\infty$ where $C\neq0$ is a real number. We shall see that $p_n\approx\sqrt{2^{-n}}$. Note that when $D_t\neq0$ it distribute as a $\sqrt2$-drifted Brownian motion. We denote it by $(B^{\sqrt{2}}_s)_{s\geq0}$ and set $t_n:=(b-a)2^{-n}$. For any $y\geq0$ let 
    \begin{equation*}
        u^n_t(y):=\pr[B^{\sqrt2}_s\text{ does not hit }2^{-n}\text{ or }n^{-1}\text{ before time }t\bigm|B_0^{\sqrt{2}}=y].
    \end{equation*}
Hence $p_n$ can be written in terms of $u^n_t(y)$ by 
\begin{equation}\label{expression of p_n}
    p_n=\int_{2^{-n}}^{2^{-n/4}}P_{t_n}(0,y)u^n_{t_n}(y)dy=\sqrt{t_n}\int_0^\infty\ind\{\frac{2^{-n}}{\sqrt{t_n}}<z<\frac{2^{-n/4}}{\sqrt{t_n}}\}P_{t_n}(0,z\sqrt{t_n})u^n_{t_n}(z\sqrt{t_n})dz.
\end{equation} 
We will apply dominated convergence theorem to (\ref{expression of p_n}). So we need to analyze each part of the integrand.
\begin{lemma}\label{convergens of u_n}
    For any $z\geq0$, we have $u_{t_n}^n(z\sqrt{t_n})\to\erf(z)$ as $n\to\infty$.
\end{lemma}
\begin{proof}
    Let $Y_s^n:=B^{\sqrt{2}}_{t_ns}/\sqrt{t_n}$. Then we have
    \begin{equation}\label{u_tn^n}
        u^n_{t_n}(z\sqrt{t_n})=\pr[Y^n_s\in(\frac{2^{-n}}{\sqrt{t_n}},\frac{n^{-1}}{\sqrt{t_n}})\text{ for all }s\in[0,1]\bigm|Y_0^n=z].
    \end{equation}
    Note that $Y_s^n=B_{t_ns}/\sqrt{t_n}+s\sqrt{2t_n}$, where $(B_t)_{t\geq0}$ is a Brownian motion without drift. So $(Y^n_s)_{s\in[0,1]}$ converges in distribution to a Brownian motion $(W_s)_{s\in[0,1]}$. Since $t_n=(b-a)2^{-n}$ we get $\frac{2^{-n}}{\sqrt{t_n}}\to0$ and $\frac{n^{-1}}{\sqrt{t_n}}\to\infty$ as $n\to\infty$. Hence by (\ref{u_tn^n}) we get 
    \begin{equation*}
        u^n_{t_n}(z\sqrt{t_n})\to\pr[W_s>0\text{ for all }s\in[0,1]\bigm|W_0=z]=\erf(z).
    \end{equation*}
    The limit can be justified by using Skorokhod representation theorem and we leave it to the Appendix.
\end{proof}
\begin{lemma}\label{p_n approx sqrt t}
    Let $p_n$ be defined as before, then $p_n\approx\sqrt{2^{-n}}$.
\end{lemma}
\begin{proof}
    By (\ref{expression of p_n}) we have 
    \begin{equation*}
        \frac{p_n}{\sqrt{t_n}}=\int_0^\infty\ind\{\frac{2^{-n}}{\sqrt{t_n}}<z<\frac{2^{-n/4}}{\sqrt{t_n}}\}P_{t_n}(0,z\sqrt{t_n})u^n_{t_n}(z\sqrt{t_n})dz.
    \end{equation*}
    Hence it is enough to show the above integral converges to a non zero number. Note that when $z>0$
    \begin{equation*}
        \begin{split}
            P_{t_n}(0,z\sqrt{t_n})&=2\sqrt{2}\left(1+z\sqrt{2t_n}+2t_n\right)e^{2\sqrt{2t_n}z}\erfc(\frac{z}{\sqrt{2}}+\sqrt{t_n})-\frac{4\sqrt{2t_n}}{\sqrt{\pi}}e^{-\frac{z^2}{2}-t_n+z\sqrt{2t_n}},
        \end{split}
    \end{equation*}
so it is easy to see from the above expression that $P_{t_n}(0,z\sqrt{t_n})\to2\sqrt{2}\erfc(\frac{z}{\sqrt{2}})$ as $n\to\infty$. Moreover, when $z\in(\frac{2^{-n}}{\sqrt{t_n}},\frac{2^{-n/4}}{\sqrt{t_n}})$ we have $z\sqrt{t_n}\to0$ and clearly $\|u^n_{t_n}\|_\infty\leq1$. Therefore, 
\begin{equation*}
    \left|\ind\{\frac{2^{-n}}{\sqrt{t_n}}<z<\frac{2^{-n/4}}{\sqrt{t_n}}\}P_{t_n}(0,z\sqrt{t_n})u^n_{t_n}(z\sqrt{t_n})\right|\leq6\sqrt{2}e^{2\sqrt{2}}\erfc(\frac{z}{\sqrt{2}})
\end{equation*}
for $n$ being large enough. Hence by dominated convergence theorem and Lemma \ref{convergens of u_n} we get
\begin{equation*}
    \frac{p_n}{\sqrt{t_n}}\to\int_0^\infty 2\sqrt{2}\erfc(\frac{z}{\sqrt2})\erf(z)dz<\infty.
\end{equation*}
Hence $p_n\approx\sqrt{t_n}\approx\sqrt{2^{-n}}$ according to the definition of $t_n$.
\end{proof}

Now we check two conditions in Proposition \ref{PM book}.
\begin{lemma}
    For any $I\in \mathcal{C}_m$, let $Z(I)$ be the Bernoulli random variable defined as before, then we have
    \begin{enumerate}[label=(\alph*)]
        \item Var$(M_n(I))\leq p_n2^{n-m+1}$,
        \item for any $\gamma>\frac{1}{2}$, $2^{n(\gamma-1)+1}p_n^{-1}\to0$ as $n\to\infty.$
    \end{enumerate}
\end{lemma}
\begin{proof}
    For (a), note that if $J_1,J_2\subset I$ that are not adjacent, then $Z(J_1)$ and $Z(J_2)$ are independent. We use the fact that $\E[Z(J_1)Z(J_2)]\leq \E[Z(J_1)]$ to get
    \begin{equation*}
        \begin{split}
            \E[M_n^2(I)]&=\sum_{J_1,J_2\in \mathcal{C}_n;J_1,J_2\subset I}\E[Z(J_1)Z(J_2)]\\
            &\leq\sum_{J_1\in \mathcal{C}_n;J_1\subset I}\sum_{J_1,J_2\text{ not adjacent}}\E[Z(J_1)]\E[Z(J_2)]
            +\sum_{J_1\in \mathcal{C}_n;J_1\subset I}\sum_{J_1,J_2\text{ adjacent}}\E[Z(J_1)]\\
            &\leq \sum_{J_1,J_2\in\mathcal{C}_n;J_1,J_2\subset I}\E[Z(J_1)]\E[Z(J_2)]+\sum_{J_1\in\mathcal{C}_n;J_1\subset I}2\E[Z(J_1)]\\
            &=\E[M_n(I)]^2+2^{n-m+1}p_n.
        \end{split}
    \end{equation*}
    Hence Var$(M_n(I))\leq 2^{n-m+1}p_n$. For (b) we have $p_n\approx\sqrt{2^{-n}}$ from Lemma \ref{p_n approx sqrt t}. So
    $2^{n(\gamma-1)+1}\sqrt{2^n}\to0$ as $n\to\infty$ provided $\gamma<\frac{1}{2}$.
\end{proof}
By Proposition \ref{PM book}, the dimension of $A$, where $A$ is the limsup fractal associated to $\{Z(I):I\in\mathcal{C}\}$ is at least $\frac{1}{2}$. This leads to the proof of lower bound in Theorem \ref{main}:
\begin{proof}
    Let $l\in\mathcal{W}^l$ starting at $(0,0)$. It is enough to show $A\subseteq T^1_l$ as $T^1_l\subseteq T$. For any $t\in A$, we can find a sequence of intervals $I_{k_1}\supseteq I_{k_2}\cdots$ such that $t\in I_{k_i}$ for all $i\in\N$, where each $I_{k_i}$ is an element in $\mathcal{C}$ with length $(b-a)2^{-k_i}$ and $k_i\to\infty$ as $i\to\infty$. Since $t\in I_{k_1}$ we can find $r_1\in\mathcal{W}^r$ such that $\frac{r_1(s)-l(s)}{\sqrt{2}}\in(\frac{1}{2^{k_1}},\frac{1}{k_1})$ for all $s\in I_{k_1}$. Choose $k_m$ such that $\frac{1}{k_m}<\frac{1}{2^{k_{1}}}$, then we can find $r_2\in\mathcal{W}^r$ such that $\frac{r_2(s)-l(s)}{\sqrt2}\in(\frac{1}{2^{k_m}},\frac{1}{k_m})$. Note that on $I_{k_m}$ we have $\frac{r_1(s)-l(s)}{\sqrt{2}}>\frac{1}{2^{k_1}}$ and $\frac{1}{2^{k_m}}<\frac{r_2(s)-l(s)}{\sqrt{2}}<\frac{1}{k_m}<\frac{1}{2^{k_1}}$. So $l(s)<r_2(s)<r_1(s)$ for all $s\in I_{k_m}$. By repeating this step, we can find a subsequence of $(k_i)_{i\geq1}$, which is denoted by $(l_i)_{i\geq1}$, such that for any $N\in\N$ and $s\in I_{l_N}$ we have $l(s)<r_N(s)<\cdots <r_1(s)$. Hence by induction we get a sequence $(r_n)_{n\geq1}\subseteq \mathcal{W}^l$ such that $(r_n(t))_{n\geq1}$ is a decreasing sequence that converges to $l(t)$ as $n\to \infty$. Thus $l(t)\in \xi^\R_t$ is not isolated from the right, which means $(l(t),t)$ is a $C_l$ point. So we get $t\in T^1_l$.
\end{proof}

\textbf{Remark}. We expect that similar methods can be applied to study the Hausdorff dimensions of the set of $C_l$ points, $C_r$ points and $C_n$ points, which also appear as an open question in \cite{SP}. We are currently investigating this direction.

\section{Appendix}
\subsection{Appendix: Transition density of $D_t$}
This section is guided by Jonathan Warren \cite[Proposition 13]{SB}. We do a general case for $D_t$, let $D_t$ be a reflected drift $\mu$ Brownian motion sticky at zero with parameter $\theta$. Then we have the following SDE description of $D_t$:
\begin{equation*}
    D_t=\int_0^t\ind\{D_s>0\}dW_s+\mu\int_0^t\ind\{D_s>0\}ds+\theta\int_0^t\ind\{D_s=0\}ds,
\end{equation*}
where $W_t$ is an standard Brownian motion. Let $P_t(0,y)$ be the transition density of $D_t$ starting at zero, we will compute the following resolvent $P_\lambda(y)$ of $P_t(0,y)$:
$$P_\lambda(y):=\int_0^\infty e^{-\lambda t}P_t(0,y)dt$$
and then invert the Laplace transform to get the transition density.
Set
\begin{equation*}
    \begin{split}
        A_t^+&=\int_0^t\ind\{D_s>0\}ds,\ \ \ \alpha_t^+=\inf\{u:A_u^+>t\},\\
        A_t^0&=\int_0^t\ind\{D_s=0\}ds,\ \ \ \alpha_t^0=\inf\{u:A_u^0>t\}.
    \end{split}
\end{equation*}
Then we have the following proposition:
\begin{proposition}\label{time change}
    There exists a reflected drift $\mu$ Brownian motion $X_t$ such that $D(\alpha_t^+)=X_t$.
\end{proposition}
\begin{proof}
   Note that there exists a standard Brownian motion $B_t$ such that $B_{A^{+}_{t}}=\int_0^t\ind\{D_s>0\}dW_s$. So we have
   \begin{equation}\label{time change D_t}
       D_t=B_{A^{+}_{t}}+\mu A_t^++\theta A^0_t.
   \end{equation}
   It is easy to see that there is no time interval $(s,u)$ such that $D_t=0$ for all $t\in(s,u)$ so we get $A_t^+$ is strictly increasing in $t$. Hence $\alpha_t^+$ is the inverse of the map $t\mapsto A_t^+$. For any $\tau\in[0,\infty)$ by setting $t=\alpha^+_\tau$ in (\ref{time change D_t}) we get 
   \begin{equation*}
       X_\tau:=D_{\alpha^+_\tau}=B_\tau+\mu\tau+\theta A^0_{\alpha^+_\tau}.
   \end{equation*}
   Since $X_\tau$ stays non-negative and $\theta A_{\alpha_\tau^+}^0$ only increases when $X_\tau=0$, we know that this is a Skorohod equation with unique solution $\theta A_{\alpha_\tau^+}^0=-\inf\{B_s+\mu s:0\leq s\leq\tau\}$, implying that $X_\tau$ is a reflected drift $\mu$ Brownian motion.
\end{proof}
Now for any $\lambda>0$, let $T_1$, $T_2$ be two independent exponential random variables with parameter $\lambda$. It is easy to see $T:=\alpha_{T_1}^0\wedge\alpha_{T_2}^+$ is also an exponential distributed with the same parameter. Hence to compute the resolvent $P_\lambda(y)$, it is enough to consider $\frac{1}{\lambda}\mathbb{P}[D_T\in dy]$. Let $M_t=-\inf\{B_s+\mu s:0\leq s\leq t\}$, then  following Lemmas are needed for evaluating the probability.
\begin{lemma}\label{M_T_2 is exp}
    Let $T_2$ be the exponential distributed random variable as before, then $M_{T_2}$ is also exponential distributed with parameter $\sqrt{2a}+\mu$, where $a=\lambda+\frac{\mu^2}{2}$. 
\end{lemma}
\begin{proof}
    According to the distribution function of $M_t$ we have
    $$\mathbb{P}[M_{T_2}>y]=e^{-2\mu y}\mathbb{P}[N(-\mu T_2,T_2)<-y]+\mathbb{P}[N(\mu T_2,T_2)>y],$$
    where $N(\gamma,\sigma^2)$ is a normal distributed random variable with mean $\gamma$ and variance $\sigma^2$. 
    The first probability is given by
    $$\mathbb{P}[N(-\mu T_2,T_2)>y]=\int_y^\infty\lambda e^{-\mu x}\int_0^\infty\frac{1}{\sqrt{2\pi t}}{\exp(-\frac{x^2}{2t}-at})dtdx.$$
    The integral with respect to $t$ can be computed through modified Bessel function  of the second kind $K_\nu(z)$ given by the following integral:
    $$K_\nu(z)=\frac{1}{2}(\frac{1}{2}z)^\nu\int_0^\infty\frac{1}{t^{\nu+1}}\exp(-\frac{z^2}{4t}-t)dt.$$
    Using the fact that $K_{-\frac{1}{2}}(z)=\sqrt{\frac{\pi}{2z}}e^{-z}$ we get
    $$\int_y^\infty\lambda e^{-\mu x}\int_0^\infty\frac{1}{\sqrt{2\pi t}}{\exp(-\frac{x^2}{2t}-at})dtdx=\int_y^\infty\frac{\lambda e^{-\mu x}}{\sqrt{2a}}e^{-\sqrt{2a}x}dx=\frac{\lambda e^{-(\sqrt{2a}+\mu)y}}{\sqrt{2a}(\mu+\sqrt{2a})}.$$
    Similarly, we can compute 
    $$\mathbb{P}[N(\mu T_2,T_2)<-y]=\int_{-\infty}^{-y}\lambda e^{\mu x}\int_0^\infty\exp(-\frac{x^2}{2t}-at)dtdx=\frac{\lambda e^{-(\sqrt{2a}-\mu)y}}{\sqrt{2a}(\sqrt{2a}-\mu)}.$$
    By putting everything together we have:
    $$\mathbb{P}[M_{T_2}>y]=\frac{\lambda}{\sqrt{2a}}\left(\frac{1}{\sqrt{2a}-\mu}+\frac{1}{\sqrt{2a}+\mu}\right)e^{-(\sqrt{2a}+\mu)y}=e^{-(\sqrt{2a}+\mu)y},$$
    which is what we desired.
\end{proof}
\begin{lemma}\label{X_T_2}
    Let $X_t$ be a reflected drift $\mu$ Brownian motion. Then
    \begin{equation*}
        \mathbb{P}[X_{T_2}\in dy]=\frac{2\lambda e^{(\mu-\sqrt{2a})y}}{\mu+\sqrt{2a}}dy,
    \end{equation*}
    where $a=\lambda+\frac{\mu^2}{2}$ as described before. 
\end{lemma}
\begin{proof}
    Firstly, we recall results from \cite{Abs}. There exists a process $Y_t$ such that $X_t$ has the same distribution as $|Y_t|$, where $Y_t$ is the unique strong solution of the SDE:
    \begin{equation*}
        dY_t=\mu \text{ sign}(Y_t)dt+dB_t.
    \end{equation*}
    And the transition density $P_{|Y|}(t,y)$ of $|Y_t|$ starting from zero is given by
    \begin{equation*}
        P_{|Y|}(t,y)=\frac{2}{\sqrt{2\pi t}}\left(e^{-\frac{(\mu t-y)^2}{2t}}-\mu e^{2\mu y}\int_y^\infty e^{\frac{-(x+\mu t)^2}{2t}}dx\right),
    \end{equation*}
    where the derivation of this formula can be found in \cite[Remark 5.2, Chapter 6]{TD} by setting $\theta=0$ in that remark. 
    Hence the probability can be computed:
    \begin{equation*}
        \begin{split}
            \mathbb{P}[X_{T_2}\in dy]&= 2e^{\mu y}\int_0^\infty \frac{\lambda}{\sqrt{2\pi t}}\exp(-at-\frac{y^2}{2t})dtdy\\&-\int_y^\infty2\mu\lambda e^{2\mu y-\mu x}dx\int_0^\infty \frac{1}{\sqrt{2\pi t}}\exp(-at-\frac{x^2}{2t})dtdy\\
            &=\frac{2\lambda}{\sqrt{2a}}e^{(\mu-\sqrt{2a})y}dy-\frac{2\mu\lambda}{\sqrt{2a}(\sqrt{2a}+\mu)}e^{(\mu-\sqrt{2a})y}dy\\
            &=\frac{2\lambda}{\sqrt{2a}+\mu}e^{(\mu-\sqrt{2a})y}dy,
        \end{split}
    \end{equation*}
    where in the second inequality we use Modified Bessel function of the second kind to compute the integral with respect to $t$.
\end{proof}
Now we back to computing $\mathbb{P}[D_T\in dy]=\mathbb{P}[D_T=0]\delta_0(y)+\mathbb{P}[D_T\neq0,D_T\in dy]$. 
Note that $A_t^0$ and $A_t^+$ increase at time $\alpha^0_t$ and $\alpha^+_t$ respectively. Hence $D_{\alpha^0_t}=0$ and $D_{\alpha^+_t}>0$ indicating that $D_T=0$ if and only if $\alpha_{T_1}^0<\alpha_{T_2}^+$. Applying the map $t\mapsto A^0_t$ to both side we get $\alpha_{T_1}^0<\alpha_{T_2}^+$ is equivalent to $\theta T_1< M_{T_2}$. So according to Lemma \ref{M_T_2 is exp} the first probability is 
\begin{equation*}
    \mathbb{P}[D_T=0]=\mathbb{P}[\theta T_1<M_{T_2}]=\frac{\lambda}{\lambda+\theta(\sqrt{2a}+\mu)}.
\end{equation*}
For the second probability, by Proposition \ref{time change}, we have
\begin{equation*}
    \mathbb{P}[D_T\neq 0,D_T\in dy]=\mathbb{P}[\theta T_1\geq M_{T_2},D_{\alpha^+_{T_2}}\in dy]=\mathbb{P}[\theta T_1\geq M_{T_2},X_{T_2}\in dy]
\end{equation*}
From the fact that $M_t$ and $X_t$ are independent (this is statement holds for a general Levy process $X_t$ and $M_t$ associated to $X_t$, see \cite[Lemma 2.1]{Ind} for details) and Lemma \ref{X_T_2}, we get
\begin{equation*}
    \pr[\theta T_1>M_{T_2},X_{T_2}\in dy]=\pr[\theta T_1>M_{T_2}]\pr[X_{T_2}\in dy]=\frac{2\lambda\theta e^{(\mu-\sqrt{2a})y}}{\lambda+\theta(\sqrt{2a}+\mu)}dy.
\end{equation*}
Hence combining everything, we get the resolvent:
\begin{equation*}
    P_\lambda(y)dy=\frac{1}{\lambda}\pr[D_T\in dy]=\frac{2\theta e^{(\mu-\sqrt{2a})y}}{\lambda+\theta(\sqrt{2a}+\mu)}dy+\frac{1}{\lambda+\theta(\sqrt{2a}+\mu)}\delta_0(y)dy.
\end{equation*}

\textbf{Remark}. By taking $\mu=0$ in the above formula, we get 
\begin{equation*}
    \pr[D_T\in dy]=\frac{2\theta\lambda e^{-\sqrt{2\lambda}y}}{\lambda+\theta\sqrt{2\lambda}}dy+\frac{\lambda}{\lambda+\theta\sqrt{2\lambda}}\delta_0(y)dy,
\end{equation*}
which is consistent  with \cite[Proposition 13]{SB}.

For the left-right Brownian motion described in (\ref{left-right BM}), we have $\theta=\mu=\sqrt{2}$, hence we get the resolvent 
\begin{equation*}
    P_\lambda(y)=\frac{2\sqrt{2} e^{(\sqrt{2}-\sqrt{2(\lambda+1)})y}}{\lambda+2\sqrt{\lambda+1}+2}+\frac{1}{\lambda+2\sqrt{\lambda+1}+2}\delta_0(y).
\end{equation*}
Inverting that Laplace transform, we get desired transition density:
\begin{equation*}
    \begin{split}
        P_t(0,y)&=2\sqrt{2}\left(1+\sqrt{2}y+2t\right)e^{2\sqrt{2}y}\erfc(\frac{y}{\sqrt{2t}}+\sqrt{t})-\frac{4\sqrt{2t}}{\sqrt{\pi}}e^{-\frac{y^2}{2t}-t+\sqrt{2}y}\\
        &+\left((1+2t)\erfc(\sqrt{t})-2e^{-t}\sqrt{\frac{t}{\pi}}\right)\delta_0(y).
    \end{split}
\end{equation*}
In fact the full transition density $P_t(x,y)$ of $D_t$ on the half space $\{(x,y):x\geq0,y\geq0\}$ can be derived through $P_t(0,y)$. Let
\begin{equation*}
    K_t(x,y):=\frac{e^{\sqrt{2}(y-x)-t}}{\sqrt{2\pi t}}\left(e^{\frac{-(y-x)^2}{2t}}-e^{\frac{(y+x)^2}{2t}}\right)
\end{equation*}
be the transition density of $\sqrt{2}$-drifted Brownian motion killed at zero. And let 
\begin{equation*}
    \tau^x:=\inf\{t:B^{\sqrt2}_t=0|B^{\sqrt2}_0=x\}
\end{equation*}
be the first hitting time of $B^{\sqrt{2}}_t$ at zero, where $B^{\sqrt{2}}_t$ is a $\sqrt2$-drifted Brownian motion. Then by strong Markov property we have
\begin{equation*}
    P_t(x,y)=K_t(x,y)+\int_0^t\pr[\tau^x\in ds]P_{t-s}(0,y).
\end{equation*}
Therefore we get
\begin{equation*}
    \begin{split}
        P_t(x,y)&=\frac{e^{\sqrt{2}(y-x)-t}}{\sqrt{2\pi t}}\left(e^{\frac{-(y-x)^2}{2t}}-e^{\frac{(y+x)^2}{2t}}\right)\\
        &+2\sqrt{2}\left(1+\sqrt{2}(x+y)+2t\right)e^{2\sqrt{2}y}\erfc(\frac{x+y}{\sqrt{2t}}+\sqrt{t})
        -\frac{4\sqrt{2t}}{\sqrt{\pi}}e^{-\frac{(x+y)^2}{2t}-t+\sqrt{2}(y-x)}\\
        &+\left((1+\sqrt{2}x+2t)\erfc(\frac{x}{\sqrt{2t}}+\sqrt{t})-2\sqrt{\frac{t}{\pi}}e^{-\frac{x^2}{2t}-\sqrt2x-t}\right)\delta_0(y).
    \end{split}
\end{equation*}

\textbf{Remark}. Another way to derive $P_t(x,y)$ is by solving its forward Kolmogorov equation. Suppose $P_t(x,y)=\rho_t(x,y)+A_t(x)\delta_0(y)$. For any $x\geq0$, the forward Kolmogorov equation is given by
\begin{equation*}
\begin{cases}
\partial_t \rho_t(x,y)=(\frac{1}{2}\partial^2_y-\sqrt{2}\partial_y)\rho_t(x,y),\quad &\text{for}\quad y>0,t>0,\\
\partial_tA_t(x)=-\sqrt{2}\rho_t(x,0)+\frac{1}{2}\partial_y\rho_t(x,0)\\
\sqrt{2}A_t(x)=\frac{1}{2}\rho_t(x,0)\\
P_0(x,y)=\delta_x(y).
\end{cases}
\end{equation*}
One can also check that the formula of $P_t(x,y)$ given above solves this PDE.

\textbf{Remark}. The key fact we used in this paper is $\int_0^\infty P_t(0,y)dy=O(\sqrt{t})$ for small $t$. This fact still holds for reflected sticky Brownian motion without drift because the drift term $\sqrt{2}t$ is of smaller order of $\sqrt{t}$ when $t$ is small. Indeed the transition density of reflected sticky Brownian motion  $Q_t(x,y)$ is computed in \cite[Section 2.4]{SB2} and one can easily check that $\int_0^\infty Q_t(0,y)dy=O(\sqrt{t})$.

\subsection{proof of Lemma \ref{speration of close time}}
\begin{proof}
    We show the case when $a=0$ and $b=1$, the general case when $a,b$ are two real numbers can be easily deduced by the same method. We wire $x_1$ and $x_2$ in binary form:
    \begin{equation*}
        x_1=\sum_{k=1}^\infty\frac{i_k}{2^k};\ x_2=\sum_{k=1}^\infty\frac{j_k}{2^k},\ i_k,j_k\in\{0,1\}.
    \end{equation*}
    Let $M_1=\inf\{k:i_k\neq j_k\}$, $M_2=\inf\{k:\frac{1}{2^k}<2\epsilon\}$ and $M=M_1\vee M_2$. Then clearly we have 
    \begin{equation*}
        x_1\in[\frac{\sum_{k=1}^M i_k 2^{M-k}}{2^M},\frac{\sum_{k=1}^M i_k2^{M-k}+1}{2^M}),\ x_2\in[\frac{\sum_{k=1}^M j_k 2^{M-k}}{2^M},\frac{\sum_{k=1}^M j_k2^{M-k}+1}{2^M}).
    \end{equation*}
    To show these two intervals are consecutive, it is enough to show $\sum_{k=1}^M(j_k-i_k)2^{M-k}=1.$
    Suppose that $M=M_1$, then we have $i_k=j_k$ for all $k<M$, hence we get desired result.
    Suppose that $M=M_2>M_1$, we claim that $j_{M_1+1}-i_{M_1+1}=-1,\cdots j_{M_2}-i_{M_2}=-1$, which will finish the proof. Note that 
    \begin{equation*}
        x_2-x_1=\frac{1}{2^{M_1}}+\frac{j_{M_1+1}-i_{M_1+1}}{2^{M_1+1}}+\cdots+\frac{j_{M_2}-i_{M_2}}{2^M_2}+\cdots.
    \end{equation*}
    Suppose by contradiction that $j_{M_1+1}-i_{M_2+1}\neq -1$, then 
    $$x_2-x_1\geq \frac{1}{2^{M_1}}-\sum_{k=M_1+2}^\infty\frac{1}{2^k}=\frac{1}{2^{M_1+1}}>\epsilon,$$
    which is a contradiction. The claim is followed by induction.
\end{proof}

\subsection{Details of Lemma \ref{convergens of u_n}}
In this section we provide details for justifying the convergence in the proof of Lemma \ref{convergens of u_n}. We shall prove the following Lemma:
\begin{lemma}
    Let $(\Omega,\mathcal{F},\pr)$ be a probability space. Suppose that $(Y_t^n)_{t\in[0,1]}$ is a sequence of stochastic process supported on $C[0,1]$ that converges in distribution to a Brownian motion $(W_t)_{t\in[0,1]}$. And let $(a_n)_{n\geq 0}$, $(b_n)_{n\geq0}$ be two sequence of positive numbers such that $a_n\to0$, $b_n\to\infty$ and $a_n<b_n$ for all $n\in\N$. Then for any $x\geq0$, we have
    \begin{equation*}
        \pr[Y_t^n\in(a_n,b_n)\text{ for all }t\in[0,1]\bigm|Y_0^n=x]\to\pr[W_t>0\text{ for all }t\in[0,1]\bigm|W_0=x]
    \end{equation*}
    as $n\to\infty$.
\end{lemma}
\begin{proof}
    By Skorokhod representation theorem, there exist a sequence of modified process $(\tilde{Y}^n_t)_{t\in[0,1]}$ and a Brownian motion $(\tilde{W}_t)_{t\in[0,1]}$ on the same probability space such that $(\tilde{Y}^n_t)_{t\in[0,1]}$ has the same distribution as $Y_t^n$ and $\|\tilde{Y}_t^n-\tilde{W}_t\|_\infty\to0$ almost surely. For any $\omega\in\Omega$ with $\|\tilde{Y}_t^n(\omega)-\tilde{W}_t(\omega)\|_\infty\to0$, note that $\tilde{W}_t(\omega)\geq0$ for all $t\in[0,1]$ is equivalent to there exists $N\in\N$ such that $\tilde{Y}^n_t\in(a_n,b_n)$ for all $t\in[0,1]$ and $n\in\N$. Moreover, the event $\{\tilde{W}_t\geq0\text{ for all }t\in[0,1]\text{ and there exists }s\in[0,1] \text{ such that }\tilde{W}_s=0\}$ has zero probability. Hence
    \begin{equation*}
        \ind\{\tilde{Y}_t^n\in(a_n,b_n)\text{ for all }t\in[0,1]\}\to\ind\{\tilde{W}_t>0\text{ for all }t\in[0,1]\},\ a.s.
    \end{equation*}
    Hence we get the desired convergence for the probability.
\end{proof}

\bibliographystyle{plain}
\bibliography{ref}
\textbf{Acknowledgments}. We are grateful to Oleg Zaboronski and Roger Tribe for a number of insightful suggestions for improving the paper.\\
\textbf{Acknowledgments}. Ruibo Kou is supported by the Warwick Mathematics Institute Centre for Doctoral Training, and gratefully acknowledges funding from the University of Warwick and the UK Engineering and Physical Sciences Research Council (Grant number: EP/W524645/1 )

\end{document}